\documentclass[a4 paper, oneside]{amsart}
\usepackage[latin1]{inputenc}
\usepackage[T1]{fontenc}
\usepackage{amsmath,amsfonts,amssymb,amsthm}
\usepackage{color}

\theoremstyle{plain}
\newtheorem{theorem}{Theorem}[section]

\newtheorem{prop}[theorem]{Proposition}
\newtheorem{cor}[theorem]{Corollary}

\theoremstyle{definition}

\theoremstyle{remark}
\newtheorem{rmk}[theorem]{Remark}

\newcommand{\set}[1]{{\left\{#1\right\}}}
\newcommand{\pa}[1]{{\left(#1\right)}}

\newcommand{\abs}[1]{{\left|#1\right|}}
\newcommand{\ra}{\rightarrow}
\newcommand{\sheaf}[2]{\mathcal{#1}_{#2}}
\newcommand{\fascio}[1]{\sheaf{O}{#1}}

\newcommand{\st}{\text{ s.t. }}

\newcommand{\pr}{\mathbb{P}}

\DeclareMathOperator{\pic}{Pic}

\title{Points of order two on theta divisors}

\author{Valeria Ornella Marcucci}
\address{Dipartimento di Matematica ``F. Casorati''\\
 	Universit\`a di Pavia\\
	via Ferrata 1, 27100 Pavia, Italy}
	\email{valeria.marcucci@unipv.it}
\author{Gian Pietro Pirola}
\address{Dipartimento di Matematica ``F. Casorati''\\
 	Universit\`a di Pavia\\
	via Ferrata 1, 27100 Pavia, Italy}
	\email{gianpietro.pirola@unipv.it}
\date{\today}


\thanks{This work has been partially supported by 1) FAR 2010 (PV) \emph{"Variet\`a algebriche, calcolo algebrico, grafi orientati e topologici"} 2) INdAM (GNSAGA) 3) PRIN 2009 \emph{``Moduli, strutture geometriche e loro applicazioni''}}

\begin{document}
\begin{abstract}
 We give a bound on the number of points of order two on the theta divisor of a principally polarized abelian variety $A$. When $A$ is the Jacobian of a curve $C$ the result can be applied in estimating the number of effective square roots of a fixed line bundle on $C$.
\end{abstract}
\maketitle

\section*{Introduction}

In this paper we give an upper bound on the number of $2$-torsion points lying on a theta divisor of a principally polarized abelian variety. Given any principally polarized abelian variety $A$ of dimension $g$ and symmetric theta divisor $\Theta\subset A$, $\Theta$ contains at least $2^{g-1}\pa{2^g-1}$ points of order two, the odd theta characteristics. Moreover, in \cite{mumequation} and \cite[Chapter IV, Section 5]{igusa} it is proved that $\Theta$ cannot contain all points of order two on $A$. 

In this work we use the projective representation of the theta group to prove the following:

\medskip

\emph{Given a principally polarized abelian variety $A$, any translated $t^*_a\Theta$ of a theta divisor $\Theta\subset A$ contains at most $2^{2g}-2^{g}$ points  of order $2$ ($2^{2g}-\pa{g+1}2^{g}$ if $t^*_a\Theta$ is irreducible and not symmetric).}

\medskip

Our bound is far from being sharp and we conjecture that the right estimate should be $2^{2g}-3^g$ as in the case of a product of elliptic curves.

When $A$ is the Jacobian of a curve $C$ the result can be applied in estimating the number of effective square roots of a fixed line bundle on $C$ (cf. Section \ref{sec:applications}).

\section{Main result}

In this section we prove our main result.

\begin{theorem}
\label{theo:theta}
Let $A$ be a principally polarized abelian variety of dimension $g$ and let $\Theta$ be a symmetric theta divisor. 
\begin{enumerate}
 \item \label{eq:firststat} For each $a \in A$ there are at most $2^{2g}-2^{g}$ points of order two lying on $t_a^*\Theta$.
\item \label{eq:secondstat} Let $a\in A$ and assume that $\Theta$ is irreducible and $t_a^*\Theta$ is not symmetric with respect to the origin. Then there are at most $2^{2g}-\pa{g+1}2^{g}$ points of order two lying on $t_a^*\Theta$.
\end{enumerate}
\end{theorem}

\begin{proof}
Denote by $\pa{K,\langle \cdot,\cdot\rangle}$  the group of $2$-torsion points on $A$ with the perfect pairing induced by the polarization. Let
\[
\set{a_1,\ldots, a_g,b_1, \ldots, b_g}
\]
be a basis of $K$ over the field of order two such that
\[
\langle a_i,b_j \rangle =\delta_{ij},\qquad \langle a_i,a_j \rangle =0, \qquad \langle b_i,b_j \rangle =0,
\]
and let 
\begin{equation}
\label{eq:H}
 H:=\langle a_1,\ldots,a_g \rangle
\end{equation}
be the subgroup of $K$ generated by the elements $a_1,\ldots, a_g$. Consider the projective morphism $\varphi\colon A \ra \mathbb{P}^{2^{g}-1}$ associated to the divisor $2\Theta$. By the construction of the projective representation of the theta group $K\pa{2\Theta}$ (see \cite{mumequation}, \cite[Chapter 4]{kempfbook} and \cite{kempfLecture}), we know that the elements of $\varphi\pa{H}$ are a basis of the projective space. In the same way, the images of the elements of a coset $H_b$ of $H$ in $K$ generate the projective space $\mathbb{P}^{2^{g}-1}$.

\medskip

Suppose by contradiction that there exists a subset $S \subset K$ such that all points of $S$ lie on  $t_a^*\Theta$ and $\abs{S}>2^{2g}-2^g$. By the previous argument, since $H_b \subset S$ for some $b$, the points of $\varphi\pa{S}$ generate the entire projective space $\mathbb{P}^{2^{g}-1}$. On the other hand, by the Theorem of the Square (\cite[Chapter II, Section 6, Corollary 4]{MumfordLibro}),
\[
t_a^*\Theta +  t_{-a}^*\Theta \equiv 2\Theta.
\]
It follows that the points of $\varphi\pa{S}$ lie on an hyperplane of $\mathbb{P}^{2^{g}-1}$. This proves \eqref{eq:firststat}.

\medskip

Now we prove the second part. Suppose by contradiction that there exists a subset $S \subset K$ such that all points of $S$ lie on  $t_a^*\Theta$ and $\abs{S}>2^{2g}-(g+1)2^g$. We claim that 
\begin{equation}
\label{eq:claim}
\tag{$*$}
\boxed{
\text{the points in $\varphi\pa{S}$ lie on a $2^g-g-2$-plane in $\pr^{2^g-1}$.}
}
\end{equation}

Given a point $\varepsilon\in S$, it holds also $\varepsilon \in t^*_{-a}\Theta$. Thus $S\subset t^*_{a}\Theta\cap t^*_{-a}\Theta$. If $t^*_{a}\Theta$ is not symmetric and irreducible, $t^*_{a}\Theta\cap t^*_{-a}\Theta$ has codimension $2$ in $A$ and we can consider the natural exact sequence
\[
 0 \ra \fascio{A}\pa{-2\Theta}  \ra \fascio{A}\pa{-t^*_{-a}\Theta}\oplus \fascio{A}\pa{-t^*_{a} \Theta} \ra I_{t^*_a\Theta\cap t^*_{-a} \Theta} \ra 0;
\]
by tensoring it with $\fascio{A}\pa{2\Theta}$ we get
\[
 0 \ra \fascio{A}  \ra \fascio{A}\pa{t^*_a\Theta}\oplus \fascio{A}\pa{t^*_{-a} \Theta} \ra I_{t^*_a\Theta\cap t^*_{-a} \Theta}\otimes \fascio{A}\pa{2\Theta} \ra 0.
\]
Passing to the corresponding sequence on the global sections, we have
\begin{multline}
 0 \ra H^0\pa{A, \fascio{A}}  \ra H^0\pa{A,\fascio{A}\pa{t^*_a\Theta}}\oplus H^0\pa{A,\fascio{A}\pa{t^*_{-a} \Theta}} \\\ra H^0(I_{t^*_a\Theta\cap t^*_{-a} \Theta}\otimes \fascio{A}\pa{2\Theta)} \ra H^1\pa{A, \fascio{A}} \ra 0,
\end{multline}
since, by the Kodaira vanishing theorem (see e.g. \cite[Chapter 1, Section 2]{GF}), 
\[
 H^1\pa{A, \fascio{A}\pa{t^*_a\Theta}}=H^1\pa{A, \fascio{A}\pa{t^*_{-a}\Theta}}=0. 
\]
It follows that 
\[
 \dim H^0(I_{t^*_a\Theta\cap t^*_{-a} \Theta}\otimes \fascio{A}\pa{2\Theta)}=g+1.
\]
Thus the points in $\varphi\pa{t^*_a\Theta\cap t^*_{-a} \Theta}$ lie on a $2^g-g-2$-plane of $\pr^{2^g-1}$ and the claim \eqref{eq:claim} is proved.

To conclude the proof of \eqref{eq:secondstat} we notice that if $\abs{S}>2^{2g}-(g+1)2^g$ then $\abs{S\cap H_b}>2^g-\pa{g+1}$ for some coset $H_b$ of $H$ (see \eqref{eq:H}). Then it follows that $\varphi\pa{S}$ contains at least $2^g-g$ independent points and we get a contradiction. 
\end{proof}

\begin{rmk}
One might expect the right bound to be $2^{2g}-3^g$ and that this is realized only in the case of a product of elliptic curves.
\end{rmk}

\begin{rmk}
The argument of Theorem \ref{theo:theta} can be also used to obtain a bound on the number of $n$-torsion points (with $n> 2$) lying on a theta divisor.
\end{rmk}

\section{Applications}
\label{sec:applications}
In this section we apply Theorem \ref{theo:theta} to the case of Jacobians. This gives a generalization of \cite[Proposition 2.5]{PrymVP}.

\begin{prop}
\label{prop:constheta}
Let $C$ be a curve of genus $g$ and $M$ be a line bundle of degree $d\leq g-1$. Given an integer $k\leq g-1-d$, for each $L\in \pic^{2k}\pa{C}$ there are at least $2^g$ line bundles $\eta\in \pic^{k}\pa{C}$ such that $\eta^2\simeq L$ and $h^0\pa{\eta \otimes M}=0$.
\end{prop}
\begin{proof}
We prove the statement for $M\simeq \fascio{C}$ and $k=g-1$. The general case follows from this by replacing $L$ with $M^2 \otimes L \otimes \fascio{C}\pa{p}^{2n}$, where $p$ is an arbitrary point of $C$ and $n:=g-1-k-d$.
Denote by $\Theta$ the divisor of effective line bundles of degree $g-1$ in $\pic^{g-1}\pa{C}$. Given the morphism
\begin{align*}
m_2\colon \pic^{g-1}\pa{C} &\ra \pic^{2g-2}\pa{C}\\
\eta &\mapsto \eta^2,
\end{align*}
we want to prove that $\abs{m_2^{-1}\pa{L} \cap \Theta}\leq 2^{2g}-2^g$. Let $\alpha\in m_2^{-1}\pa{L}$, we have
\[
m_2^{-1}\pa{L}=\set{\alpha\otimes\sigma \st \sigma^2=\fascio{C}}.
\]
If $\abs{m_2^{-1}\pa{L} \cap \Theta}> 2^{2g}-2^g$, then there are more than $2^{2g}-2^g$ points of order two lying on a translated of a symmetric theta divisor of $J\pa{C}$ and, by \eqref{eq:firststat} of Theorem \ref{theo:theta}, we get a contradiction.
\end{proof}

\begin{rmk}
If we apply Proposition \ref{prop:constheta} to $M=\fascio{C}, L=\omega_C$, we get that on a curve of genus $g$ there are at most $2^{2g}-2^g$ effective theta characteristics. We notice that when $g=2$ they are the $6$ line bundles of type $\fascio{C}\pa{p}$ where $p$ is a Weierstrass point. When $g=3$ and $C$ is not hyperelliptic, they correspond to the $28$ bi-tangent lines to the canonical curve.
\end{rmk}

\begin{cor}
\label{cor:ptiord2perapplicazione}
Let $C$ be a curve of genus $g$ and $M_1,\ldots M_N$ be a finite number of line bundles of degree $d\leq g-1$. Given an integer $k \leq g-1-d$, if $\eta$ is a generic line bundle of degree $k$ such that $h^0\pa{\eta^2}>0$, then
\[
h^0\pa{\eta \otimes M_i}=0 \qquad \forall i=1,\ldots, N.
\]
\end{cor}
\begin{proof}
Let
\[
\Lambda:=\set{\eta \in \pic^{k}\pa{C}: h^0\pa{\eta^2}>0},
\]
and, for each $i=1,\ldots N$, consider its closed subset
\[
\Lambda_i:=\set{\eta \in \Lambda: h^0\pa{M_i\otimes\eta}>0}.
\]
We remark that $\Lambda$ is a connected $2^{2g}$-\'etale covering of the image of the $2k$-th symmetric product of $C$ in $\pic^{2k}\pa{C}$. By Proposition \ref{prop:constheta}, for each effective $L\in \pic^{2k}\pa{C}$ there exists $\eta \in \Lambda \setminus \Lambda_i$ such that $\eta^2\simeq L$. It follows that $\Lambda_i$ is a proper subset of $\Lambda$. Since $\Lambda$ is irreducible, also the set
\[
\bigcup_{i=1}^N\Lambda_i=\set{\eta\in \pic^{k}\pa{C}: h^0\pa{M_i\otimes\eta}>0 \mbox{ for some $i$}}
\]
is a proper closed subset of $\Lambda$.
\end{proof}

\bibliographystyle{alpha}

\bibliography{biblio}

\begin{thebibliography}{Mum08}

\bibitem[GH94]{GF}
P.~Griffiths and J.~Harris.
\newblock {\em Principles of algebraic geometry}.
\newblock Wiley Classics Library. John Wiley \& Sons Inc., New York, 1994.
\newblock Reprint of the 1978 original.

\bibitem[Igu72]{igusa}
J.~Igusa.
\newblock {\em Theta functions}.
\newblock Springer-Verlag, New York, 1972.
\newblock Die Grundlehren der mathematischen Wissenschaften, Band 194.

\bibitem[Kem89]{kempfLecture}
G.~R. Kempf.
\newblock The addition theorem for abstract theta functions.
\newblock In {\em Algebraic geometry and complex analysis ({P}\'atzcuaro,
  1987)}, volume 1414 of {\em Lecture Notes in Math.}, pages 1--14. Springer,
  Berlin, 1989.

\bibitem[Kem91]{kempfbook}
G.~R. Kempf.
\newblock {\em Complex abelian varieties and theta functions}.
\newblock Universitext. Springer-Verlag, Berlin, 1991.

\bibitem[MP]{PrymVP}
V.~Marcucci and G.~P. Pirola.
\newblock Generic {T}orelli theorem for {P}rym varieties of ramified coverings.
\newblock {\em Compositio Math.}
\newblock to appear, arXiv:1010.4483v3.

\bibitem[Mum66]{mumequation}
D.~Mumford.
\newblock On the equations defining abelian varieties. {I}.
\newblock {\em Invent. Math.}, 1:287--354, 1966.

\bibitem[Mum08]{MumfordLibro}
D.~Mumford.
\newblock {\em Abelian varieties}, volume~5 of {\em Tata Institute of
  Fundamental Research Studies in Mathematics}.
\newblock Published for the Tata Institute of Fundamental Research, Bombay,
  2008.

\end{thebibliography}
\end{document}